\documentclass{amsart}

\usepackage{amsmath}

\usepackage{amsfonts}

\usepackage{hyperref}

\usepackage{latexsym}

\usepackage{amssymb}

\usepackage[ansinew]{inputenc}

\usepackage{amscd}

\newtheorem{theorem}{Theorem}[section]
\newtheorem{theoreme}{Theorem}
\newtheorem{lemma}[theorem]{Lemma}
\newtheorem{corollary}[theorem]{Corollary}
\newtheorem{proposition}[theorem]{Proposition}
\newtheorem{definition}[theorem]{Definition}

\newtheorem{remark}[theorem]{Remark}

\newcommand{\ma}[1]{\ensuremath{\mathbb{#1}}}
 
\font\bb=msbm7 at 10 pt

\def \C {\hbox{\bb C}}
\def \Z {\hbox{\bb Z}}
\def \Q {\hbox{\bb Q}}
\def \N {\hbox{\bb N}}
\def \PP {\hbox{\bb P}}

\def \F {\hbox{\bb F}}
\def \R {\hbox{\bb R}}
\def \A {\hbox{\bb A}}

\def \d {\textbf{d}}

\def \T {\mathcal{T}}

\def \AA {\mathcal{A}}

\def \B {\mathcal{B}}

\def \H {\mathcal{H}}

\def \O {\mathcal{O}}
\def \K {\mathcal{K}}
\def \Q {\hbox{\bb Q}}
\def \I {\mbox{\bf I}}

\def \II {\mathcal{I}}
\def \del {\boldsymbol\delta}
\def \sgn {\mbox{\rm{sgn}}}

\newcommand{\1}{\ensuremath{\textrm{\bf 1}}}
\newcommand{\ee}{\ensuremath{\textrm{\bf e}}}
\newcommand{\dd}{\ensuremath{\textrm{\bf d}}}
\newcommand{\ii}{\ensuremath{\textrm{\bf i}}}
\newcommand{\jj}{\ensuremath{\textrm{\bf j}}}
\newcommand{\uu}{\ensuremath{\textrm{\bf u}}}

\newcommand{\nn}{\ensuremath{\textrm{\bf n}}}
\newcommand{\X}{\ensuremath{\textbf{X}}}
\def \x {\mbox{\textbf{x}}}

\newcommand{\Ima}{\ensuremath{\mbox{\rm{Im }}}}
\newcommand{\Gal}{\ensuremath{\mbox{\rm{Gal}}}}

\author{R\'egis Blache}
\address{\'Equipe LAMIA,
IUFM de la Guadeloupe}
\email{rblache@iufm.univ-ag.fr}

\title[Congruences for $L$-functions]{Congruences for $L$-functions of additive exponential sums}

\begin{document}

\begin{abstract}
We give a congruence for $L$-functions coming from affine additive exponential sums over a finite field. Precisely, we give a congruence for certain operators coming from Dwork's theory. This congruence is very similar to the congruence of Manin for the characteristic polynomial of the action of Frobenius on the Jacobian of a curve defined over a finite field.
\end{abstract}


\maketitle

\section*{Introduction}

In a classical paper \cite{manin}, Manin gives a congruence for the characteristic polynomial of the Frobenius endomorphism on the jacobian of a curve $C$, defined over the finite field $k:=\F_{q}$, $q=p^m$, in terms of its Hasse-Witt matrix $A$. Let $g$ denote the genus of $C$, and $\tau$ the $p$-th power morphism. We transform slightly Manin's result in order to give a congruence for the numerator $L(C,T)$ of the zeta function of the curve $C$ (it is the reciprocal of the characteristic polynomial of Frobenius)
$$L(C,T)\equiv \det\left(\I_g-AA^\tau\cdots A^{\tau^{m-1}}T\right)\mod p.$$
This result gives a congruence modulo $p$ for the coefficients of the polynomial $L(C,T)$, along the horizontal slope of its Newton polygon.

Our aim here is to give a congruence similar to Manin's, valid for any $L$-function associated to additive exponential sums over affine space. 

For any $r\geq 1$, denote by $k_r$ the degree $r$ extension of the field $k$ inside a fixed algebraic closure $\overline{k}$. Choose a non trivial additive character $\psi$ of $\F_p$; using the trace, it extends to a character $\psi_{mr}$ for each $r\geq 1$. Fix a subset $D\subset \N^n$ which is not contained in any of the coordinate hyperplanes, and a $n$ variable polynomial $f\in k[\x]$ having its exponents in $D$. One can define the exponential sum $S_r(f)$ over the $k_r$ rational points of affine space $\A^n$. As usual, one defines a generating series from these sums when $r$ varies; this is the $L$-function associated to $f$ and $\psi$ which we denote by $L(\A^n,f;T)$ (we omit the character $\psi$ since it is fixed once and for all)
$$L(\A^n,f;T):=\exp\left(\sum_{r\geq 1} S_r(f)\frac{T^r}{r}\right).$$

It is known that this function is rational; actually it lies in $\Q_p(\zeta_p)(T)$, and has its coefficients in the ring $\Z_p[\zeta_p]$. We denote by $\pi$ the root in $\C_p$ of the polynomial $X^{p-1}+p$ such that $\psi(1)\equiv 1+\pi \mod \pi^2$. This is a generator of the maximal ideal of $\Z_p[\zeta_p]$. A congruence modulo $\pi$ for the $L$-function would be trivial, since all exponential sums have positive valuation. In other words all reciprocal roots and poles of the $L$-function have positive valuation. In \cite[Theorem 2.1]{bl}, we show that the $p$-density $\delta:=\delta_{p}(D)$ of the set $D$ is a lower bound for these ($q$-adic) valuations, and that it is tight, in the sense that there exists at least one polynomial in $\overline{k}[D]$ whose $L$-function has one reciprocal root or pole of exact valuation $\delta$. As a consequence, if we consider the $L$-function as an element in $\Z_p[\zeta_p][[T]]$, it lies in the subring
$$M_{\delta}:=\{\sum_{i\geq 0} a_iT^i,~a_i\in \Z_p[\zeta_p],~v_q(a_i)\geq \delta i\},$$
and we shall get a non trivial congruence if we consider it modulo the ideal
$$I_{\delta}:=\{\sum_{i\geq 0} a_iT^i,~a_i\in \Z_p[\zeta_p],~v_q(a_i)> \delta i\}.$$

This is our purpose; in order to write down the result, we need some notations. For any subset $I\subset\{1,\ldots,n\}$, denote by $D_I$ the subset of $D$ consisting of elements $\dd(d_1,\cdots,d_n)$ such that $d_j= 0$ for any $j\notin I$. Fix a polynomial $f(\x):=\sum_D c_{\dd} \x^{\dd} \in k[\x]$. Consider the Teichmüller liftings $\gamma_{\dd}$ of the $c_{\dd}$ in the unramified extension of degree $m$ of $\Q_p$, set $\Gamma:=(\gamma_{\dd})_{\dd\in D}$ and $\Gamma_I:=(\gamma_{\dd})_{\dd\in D_I}$. For each $I$, we can define its $p$-density $\delta_p(D_I)$ from \cite{bl}; in the following we define some new invariants associated to the set $D_I$ and the prime $p$. A subset of $\N_I$ which we call the $p$-minimal support, and a $N_I\times N_I$ matrix $M(\Gamma_I):=M_{D_I,p}(\Gamma_I)$ whose coefficients lie in $\Z_p[\Gamma_I]$. With this at hand, we have

\begin{theoreme} 
\label{Lcong}
The $L$-function defined from the exponential sums $S_r(f)$ satisfies the following congruence in $M_\delta$

$$L(\A^n,f;T)\equiv \prod \det\left(\I_{N_I}-\pi^{m(p-1)\delta}TM(\Gamma_I)^{\tau^{m-1}}\cdots M(\Gamma_I)\right)^{(-1)^{\#I+1}} \mod I_{\delta}$$
where the product is over those $I$ such that $\delta_p(D_I)+n-\#I=\delta_p(D)$. 

\end{theoreme}

Note that as long as the density is a rational number, the number $m(p-1)\delta$ is not necessarily an integer and the number $\pi^{m(p-1)\delta}$ is not in $\Z_p[\zeta_p]$ in general; but the rational function on the right hand side has its coefficients in $\Z_p[\zeta_p]$.

Note also that in many cases (for instance when $D$ has an element with all its coordinates positive), there is no subset $I\subset\{1,\ldots,n\}$ such that $\delta_p(D_I)+n-\#I=\delta$. Then the right hand side of the congruence is a polynomial. 

\medskip

From the orthogonality relations on additive characters, the zeta functions of projective and affine varieties are expressed as $L$-functions. In this way one can recover some classical results about these zeta functions. In the case of a projective hypersurface $V$ having equation $F=0$ of degree $d$ in $\PP^n$, with $d\geq n$, Miller \cite[Corollaire 1]{mi} gives an expression of the matrix of the Cartier operator acting on the space $H^0(V,\Omega^{n-1})$ from the coefficients of $F$. One can show that if we set $f:=yF$, the matrix in Theorem \ref{Lcong} boils down to (the transpose of) Miller's one. There are also congruences modulo $p$ for the zeta function of a variety in \cite{ka2}; but these congruences are trivial when the Hodge number $h^{n-1,0}$ is zero, i.e. when the Newton polygon of the interesting part of the zeta function has no horizontal slope. The result above allows one to give non-trivial congruences in any case.

\medskip

We emphasize the case $n=1$. This case has been our starting point, since it has drawn much attention, in connection with Artin-Schreier curves, i.e. $p$-cyclic coverings of the projective line ramified exactly at one point. These curves have $p$-rank $0$, and the right-hand side of Manin's congruence is $1$ in this case. This question has been studied in \cite{sz3}; the method there is to compute the Verschiebung action on the first de Rham cohomology space of a curve by taking power series expansions at a rational point, and then use Katz's sharp slope estimate \cite{ka}; this gives the first vertex of the Newton polygon of the numerator of the zeta function when the characteristic $p$ is large compared with the degree. These calculations have been applied to the case $p=2$ in \cite{sz1,sz2}, where the first slope is given in many cases.

The congruence above, applied in the one variable case, extends and generalizes these results in the following way. If $f\in k[x]$ is a polynomial having its exponents in $D\subset \N$, and $C$ is the (Artin-Schreier) curve having equation $y^p-y=f$, one can express the numerator $L(C,T)$ of the zeta function of the curve $C$ as a product of $L$-functions \cite[p95]{bo}. We get the following congruence for the numerator (note it is the reciprocal of the characteristic polynomial of Frobenius action)

$$L(C,T)\equiv N_{\ma{Q}_p(\zeta_p)/\ma{Q}_p}\left(\det\left(\I_{N}-\pi^{m(p-1)\delta}TM(\Gamma)^{\tau^{m-1}}\cdots M(\Gamma)\right)\right) \mod I_{\delta}$$

We shall return to the question of Artin-Schreier curves in a forthcoming article.

\medskip

Our methods are very close to Dwork's original one. In section \ref{dwor}, we express the $L$-functions as alternating products of Fredholm determinants coming from $p$-adic completely continuous operators \cite{se}; this follows rather closely the method of Katz \cite{ka1} for the number of points of varieties. Then we give technical results about minimal solutions of certain modular equations along section \ref{irrsol}; there we define the building blocks for the matrices $M_{D,p}$. Finally we examine closely the coefficients of the Fredholm determinants in section \ref{secmain}, and link their principal parts to minimal solutions defined above. Putting these results together gives Theorem \ref{Lcong}, that we prove at the end of the paper.

\section{Dwork's trace formula}
\label{dwor}

In this section we briefly describe the tools that we shall use throughout the paper. Everything could be described in terms of rigid cohomology, but for sake of simplicity we shall adopt the point of view of Robba \cite{ro} (note that we do not use cohomology, since we work with nuclear matrices), which has the benefit to be explicit. Our aim in this section is to give different expressions for the $L$-functions we consider, in terms of some Fredholm determinants. This is rather classical, and the only new result (to our knowledge) is Lemma \ref{coeffssplit}, which gives an expression for the principal parts of the coefficients of some of Dwork's splitting functions, very close to the well known Stickelberger's congruence for Gauss sums.

\subsection{The splitting functions}

We denote by $\Q_p$ the field of $p$-adic numbers, and by $\K_m:=\Q_p(\zeta_{q-1})$ its 
(unique up to isomorphism) unramified extension of degree $m$. Let 
$\O_m=\Z_p[\zeta_{q-1}]$ be the valuation ring of $\K_m$; the elements of finite order in 
$\O_m^\times$ form a group $\T_m^\times$ of order $p^m-1$, and 
$\T_m:=\T_m^\times\cup\{0\}$ is the {\it Teichm\"uller} of $\K_m$. Note that 
it is the image of a section of reduction modulo $p$ from $\O_m$ to its 
residue field $\F_q$, called the {\it Teichm\"uller lift}. Let $\tau$ be the Frobenius; it is the 
generator of $\Gal(\K_m/\Q_p)$ which acts on $\T_m$ as the $p$th power 
map. Finally we denote by $\C_p$ a completion of a fixed algebraic 
closure $\overline{\Q}_p$ of $\Q_p$.

\medskip

Let $\pi \in \C_p$ be the root of the polynomial $X^{p-1}+p$ defined in the introduction. It is well known that $\Q_p(\pi)=\Q_p(\zeta_p)$ is a totally ramified extension of degree $p-1$ of $\Q_p$. We shall 
frequently use the valuation $v:=v_\pi$, normalized by $v_\pi(\pi)=1$, 
instead of the usual $p$-adic valuation $v_p$, or the $q$-adic valuation 
$v_q$. 

\medskip

We define the power series $\theta(X):=\exp(\pi X-\pi X^p)$; this is a {\it splitting function} in Dwork's terminology \cite[p55]{dw}. Its values at the points of $\T_1$ are $p$-th roots of unity; actually this function represents the additive character $\psi$. It is well known that 
$\theta$ converges for any $x$ in $\C_p$ such that 
$v_p(x)>-\frac{p-1}{p^2}$. We also define 
$$\theta_m(X):=\prod_{i=0}^{m-1} \theta(X^{p^i})=\exp(\pi X-\pi X^q):=\sum_{n\geq 0} \lambda_n^{(m)} X^n.$$

We need a precise estimate for the valuations of the coefficients of the series $\theta_m$. Let us introduce some notations.

\begin{definition}
For $n$ a non negative integer, we denote by $s_p(n)$ the {\rm $p$-weight} of $n$ : in other words, if $n=n_0+pn_1+\cdots+p^tn_t$ with $0\leq n_i\leq p-1$, we have $s_p(n)=n_0+\dots+n_t$.
Moreover we set $n!!:=n_0!\dots n_{t-1}!$.
\end{definition}

We give an expression for the principal parts of the coefficients of Dwork's splitting functions defined above (compare Stickelberger's theorem for Gauss sums \cite{st})

\begin{lemma}
\label{coeffssplit}
Notations are as in the definition above. In the ring $\Z_p[\zeta_p]$, we have the following congruences for the coefficients of the splitting function $\theta_m$
 $$\lambda^{(m)}_n \equiv \left\{ \begin{array} {rcl} \frac{\pi^{s_p(n)}}{n!!} \mod \pi^{s_p(n)+p-1}  & \mbox{\rm if} &  0\leq n\leq q-1;\\
0 \mod \pi^{s_p(n)+p-1} & \mbox{\rm if} & n\geq q. \\
\end{array}
\right.$$

\end{lemma}

\begin{proof}
Recall that $\theta_m(X)=\exp(\pi X-\pi X^q)$. From the well known expansion $\exp X=\sum_{n\geq 0} \frac{X^n}{n!}$, we get the expression
$$\lambda^{(m)}_n=\sum_{r,s,r+qs=n} (-1)^s \frac{\pi^{r+s}}{r!s!}.$$

Assume $0\leq n\leq q-1$; then we get $\lambda^{(m)}_n=\frac{\pi^{n}}{n!}$. From a result of Anton, we have the congruence 
$$n!\equiv (-p)^a n!! \mod p^{a+1},\quad a=\frac{n-s_p(n)}{p-1},$$
which gives the result for $0\leq n\leq q-1$ (recall that $\pi^{p-1}=-p$).

Assume $n\geq q$, and write $n=n_0+pn_1+\cdots+p^tn_t$, with $t\geq m$. First observe that, from above,
$$\frac{\pi^{r+s}}{r!s!}\equiv \frac{\pi^{s_p(r)+s_p(s)}}{r!! s!!} \mod \pi^{s_p(r)+s_p(s)+p-1}$$
From the expression $n=r+qs$, and since $s_p(q)=1$, we deduce $s_p(n)\leq s_p(r)+s_p(s)$, and $s_p(n)\equiv s_p(r)+s_p(s) \mod p-1$. The formula $s_p(n)= s_p(r)+s_p(s)$ holds if and only if $0\leq s_i \leq n_{m+i}$ for any $0\leq i\leq t-m$. Thus we get
$$\lambda^{(m)}_n\equiv \left(\sum_{s_0=0}^{n_m}\cdots \sum_{s_{t-m}=0}^{n_t} (-1)^s\frac{1}{r!!s!!}\right) \pi^{s_p(n)}\mod \pi^{s_p(n)+p-1}.$$

Assume $p$ is odd. From the identities $r!!=n_0!!\dots n_{m-1}!!(n_m-s_0)!!\cdots (n_t-s_{m-t})!!$ and $(-1)^s=(-1)^{s_0+\cdots+s_{t-m}}$, one can rewrite the multiple sum as
$$\frac{1}{n!!}\sum_{s_0=0}^{n_m}\frac{(-1)^{s_0}n_m!}{(n_m-s_0)!s_0!}\cdots \sum_{s_{t-m}=0}^{n_t} \frac{(-1)^{s_{t-m}}n_t!}{(n_t-s_{t-m})!s_{t-m}!}.$$
Since each sum is an alternate sum of binomial coefficients, it is zero, and we get the result.

\medskip
 
When $p=2$, we just have to remove the signs; now the sums of binomial coefficients are powers of $2$, and we get the result.
\end{proof}

\subsection{Dwork's trace formula}

Consider the space $\H^\dagger(\AA)$ of overconvergent functions over the unit ball $\AA=\{(x_1,\ldots,x_n) \in \C_p^n,~\max_i |x_i|\leq 1\}$ of $\C_p^n$, with basis $\B\{\X^{\ii},~\ii\in \N^n\}$. Let $D$ denote a finite subset of $\N^n$, and set $f(\x):=\sum_{\dd\in D} c_{\dd}\x^{\dd}$, $c_{\dd}\in k$. Denote by $\gamma_{\dd}$ the Teichmüller lift of $c_{\dd}$. We set $\Gamma:=(\gamma_{\dd})_{\dd\in D}$ in the following.

From Dwork's splitting functions, we can define two elements in $\H^\dagger(\AA)$ in the following way
$$F_1(\Gamma,\X)=\prod_{\dd\in D} \theta(\gamma_{\dd}\X^{\dd}),$$
$$F_m(\Gamma,\X)=\prod_{i=0}^{m-1} F_1(\Gamma,\X^{p^i})^{\tau^i}=\prod_{\dd\in D} \theta_m(\gamma_{\dd}\X^{\dd}),$$

We also define an operator $\Psi$ over $\H^\dagger(\AA)$ by $\Psi\left( \sum_{\ii\in \ma{N}^n} b_{\ii}\X^{\ii}\right)=\sum_{\ii\in \ma{N}^n} b_{p\ii}\X^{\ii}$ (it is Dwork's application in Robba's terminology). We consider the nuclear operator $\alpha(\Gamma)$ defined over $\H^\dagger(\AA)$ by $\alpha(\Gamma):=\Psi^m\circ F_m(\Gamma,\X)$, where a series (here $F_m(\Gamma,\X)$) denotes multiplication by this series on $\H^\dagger(\AA)$. We can factorize $\alpha(\Gamma)$ in terms of the semi-linear (actually $\Q_p(\pi)$-linear, but not $\Q_q(\pi)$-linear) operator $\beta(\Gamma):=\tau^{-1}\circ \Psi \circ F_1(\Gamma,\X)$, simply as $\alpha(\Gamma)=\beta(\Gamma)^m$.

Finally, in order to give an expression for the $L$-function, we define a deRham type complex $(\Omega^\bullet,d_\bullet)$. Set $d\X_I=dX_{i_1}\wedge\cdots\wedge dX_{i_k}$ for $I=\{i_1,\cdots,i_k\}$ with $i_1<\cdots<i_k$, and for any $1\leq i\leq n$, consider the differential on $\H^\dagger(\AA)$ defined by $D_i=\frac{\partial}{\partial X_i}+H_i$, where we have set $H_i= \frac{1}{F_m}\frac{\partial F_m}{\partial X_i}$. Then the complex is defined by 
$$\Omega^k=\bigoplus_{\stackrel{I\subset \{1,\ldots,n\}}{\# I=k}} \H^\dagger(\AA)d\X_I$$ 
for any $0\leq k\leq n$, with boundary operator $d_k:\Omega^k\mapsto \Omega^{k+1}$ defined by 
$$d_k(f(\X)d\X_I)=(\sum_{i=1}^n D_i(f)dX_i)\wedge d\X_I.$$
We extend the operator $\alpha(\Gamma)$ to an operator $\alpha(\Gamma)_\bullet$ of this complex, setting $$\alpha(\Gamma)_k(f(\X)d\X_I):=q^{n-k}\frac{1}{\X_I}\alpha(\Gamma)(\X_If(\X))d\X_I,$$ 
where $\X_I$ denotes the monomial $X_{i_1}\dotsm X_{i_k}$. Moreover the operators $\alpha(\Gamma)_k$ are all nuclear operators, thus they have well-defined trace and Fredholm determinant \cite{se}.

As an application of Dwork's trace formula, we get the following expression for the $L$-function
\begin{equation}
\label{dwtr}
L(\A^n,f;T)=\prod_{k=0}^n \det(\I-T\alpha(\Gamma)_k|\Omega^k)^{(-1)^{k+1}}.
\end{equation}

\subsection{Decomposition in Fredholm determinants}
\label{matrix}

From now on we consider the matrix $A(\Gamma)$ ({\it resp.} $B(\Gamma)$) of the operator $\alpha(\Gamma)$ ({\it resp.} $\beta(\Gamma)$) with respect to the basis $\B$. It is an easy calculation to check that if we set $F_m(\Gamma,\X):= \sum_{\ii\in \ma{N}^n} f^{(m)}_{\ii}(\Gamma)\X^{\ii}$ ({\it resp.} $F_1(\Gamma,\X):= \sum_{\ii\in \ma{N}^n} f^{(1)}_{\ii}(\Gamma)\X^{\ii}$), then the $(\ii,\jj)$ coefficient of $A(\Gamma)$ ({\it resp.} $B(\Gamma)$) is $f^{(m)}_{q\ii-\jj}(\Gamma)$ ({\it resp.} $f^{(1)}_{p\ii-\jj}(\Gamma)$).

\begin{definition}
\label{defNI}
Let $I$ be a subset of $\{1,\ldots,n\}$, possibly empty; in the following, we denote by $|I|$ its cardinality. 
\begin{itemize}
	\item[(i)] For any $\ii(i_1,\ldots,i_n)$ in $\N^n$, define its \emph{support} $[\ii]$ as $\{k,i_k\neq 0\}$.
	\item[(ii)] We define the matrices $A(\Gamma)_I$, $A(\Gamma)^I$ and $B(\Gamma)_I$ by 
$$A(\Gamma)_I:=(f^{(m)}_{q\ii-\jj}(\Gamma))_{[\ii],[\jj]=I},~A(\Gamma)^I:=(f^{(m)}_{q\ii-\jj}(\Gamma))_{[\ii],[\jj]\supset I},~B(\Gamma)_I:=(f^{(1)}_{p\ii-\jj}(\Gamma))_{[\ii],[\jj]=I}.$$
	\item[(iii)] Define $\N_I$ as $\{\ii\in \N^n,~[\ii]\subset I\}$, and $D_I=D\cap \N_I$.

	\item[(iv)] Finally, define the operator $\del$ acting on formal power series with coefficients in $\C_p$ and constant coefficient equal to $1$ by $g^{\del}(t):=\frac{g(t)}{g(qt)}$.
\end{itemize}
\end{definition} 

\begin{remark}
Note that for $A(\Gamma)$ a nuclear matrix, every matrix $A(\Gamma)_I$ or $A(\Gamma)^I$ is also nuclear. 

Note also that $A(\Gamma)_\emptyset$ is the $1\times 1$ matrix having coefficient $1$, and $A(\Gamma)^\emptyset=A(\Gamma)$.

Finally, we have the relation $A(\Gamma)=B(\Gamma)^{\tau^{m-1}}B(\Gamma)^{\tau^{m-2}}\ldots B(\Gamma)$ from the factorisation $\alpha(\Gamma)=\beta(\Gamma)^m$ of the operator $\alpha(\Gamma)$ in terms of the semi-linear operator $\beta(\Gamma)$, 
\end{remark}

Let us give an expression of the $L$-function from the Fredholm determinants of the matrices we have just defined

\begin{lemma} 
\label{exp}
Let $I,J$ denote subsets of $\{1,\ldots,n\}$, possibly empty. We have the following expressions 
\begin{itemize}
	\item[(i)] $\det(\I-TA(\Gamma)^I)=\prod_{J\supset I}\det(\I-TA(\Gamma)_J)$;
	\item[(ii)] $\det(Id-T\alpha(\Gamma)_k|\Omega^k)=\prod_{\# I=k }\det(\I-q^{n-k}TA(\Gamma)^I)$;
	\item[(iii)]	$L(\A^n,f;T)=\prod_J \det(\I-q^{n-|J|}TA(\Gamma)_J)^{-(-\del)^{|J|}}$, where the product is over all subsets of $\{1,\ldots,n\}$, including $\emptyset$;
	\item[(iv)] the matrix factorisation $A(\Gamma)_J=B(\Gamma)_J^{\tau^{m-1}}B(\Gamma)_J^{\tau^{m-2}}\ldots B(\Gamma)_J$ remains true;
	\item[(v)] if $D_I\subset \N_J$, for some $J\subsetneq I$, then $A(\Gamma)_I$ is the zero matrix. 
\end{itemize}
\end{lemma}

\begin{proof}
On one hand, the degree $t$ coefficient of the Fredholm determinant $\det(\I-TA(\Gamma)^I)$ is the sum of the terms
$$M_\sigma=\sgn(\sigma)\prod_{u=1}^t f^{(m)}_{q\ii_{u}-\sigma(\ii_{u})}(\Gamma)$$
when $\ii_1<\ldots<\ii_t$ (with respect to lexicographic order) runs over elements in $\N^n$ whose support contains $I$ and $\sigma$ over the symmetric group on $\{\ii_1,\ldots,\ii_t\}$.

On the other hand, the degree $t$ coefficient of the product $\prod_{I\subset J}\det(\I-TA(\Gamma)_J)$ is the sum of the terms above with $[\sigma(\ii_{u})]=[\ii_u]$ for any $\ii_u$. Reordering the factors, the term $M_\sigma$ can be written as a product of terms of the form  $M=\prod_{v=0}^{l_i-1} f^{(m)}_{q\sigma^v(\ii_{u})-\sigma^{v+1}(\ii_{u})}(\Gamma)$, corresponding to the disjoint cycles whose product is $\sigma$.

Remark that we have $f^{(m)}_{q\ii-\jj}(\Gamma)=0$ when $[\jj]\nsubseteq[\ii]$. In order for the term $M$ above to be non-zero, we must have $[\ii_{u}]=[\sigma^{l_i}(\ii_{u})]\subseteq\ldots \subseteq [\ii_{u}]$, and the permutation $\sigma$ is a product of cycles preserving supports. This is the first assertion.

We turn to assertion {\it ii/}. It is sufficient to show that for any $I\subset \{1,\ldots,n\}$, $\# I=k$, the restriction of $\alpha(\Gamma)_k$ to $\H^\dagger(\AA)d\X_I$ has matrix $q^{n-k}A(\Gamma)^I$ with respect to some basis. We denote by $\1_I$ the vector in $\{0,1\}^n$ having support $I$. Consider the basis $\{\X^{\ii}d\X_I,~\ii\in \N^n\}$ of $\H^\dagger(\AA)d\X_I$; from the definition of $\alpha(\Gamma)_k$, we have 
$$\alpha(\Gamma)_k(\X^{\jj}d\X_I)=q^{n-k}\sum_{\ii\in \ma{N}^n} f^{(m)}_{q(\ii+\1_I)-(\jj+\1_I)}(\Gamma)\X^{\ii}d\X_I,$$
and the matrix of the restriction of $\alpha(\Gamma)_k$ to $\H^\dagger(\AA)d\X_I$ with respect to the above basis is $A(\Gamma)_I':=(q^{n-k}f^{(m)}_{q(\ii+\1_I)-(\jj+\1_I)}(\Gamma))_{\ii,\jj\in \N^n}$. Now the map $\ii\mapsto \ii+\1_I$ is a bijection from $\N^n$ to the set $\{\ii\in \N^n,~I\subset [\ii]\}$, and we get $A(\Gamma)_I'=q^{n-k}A(\Gamma)^I$.

The third assertion is a consequence of the first two, and the expression (\ref{dwtr}) of the $L$-function; using (ii), then (i), we get the expressions
\begin{eqnarray}
L(\A^n,f;T) & = & \prod_{k=0}^n \prod_{\# I=k }\det(\I-q^{n-k}TA(\Gamma)^I)^{(-1)^{k+1}}\\
            & = & \prod_{k=0}^n \prod_{\# I=k } \prod_{J\supset I}\det(\I-q^{n-k}TA(\Gamma)_J)^{(-1)^{k+1}} \\
       \end{eqnarray}
We can exchange the products, in order to begin with the product over $J$; the factor $\det(\I-q^{n-k}TA(\Gamma)_J)^{(-1)^{k+1}}$ appears once for each subset of $J$ having cardinality $k$, i.e. $\binom{\# J}{k}$ times. Now assertion (iii) follows from the expression
$$f^{(-\del)^k}(t)=\prod_{i=0}^{k} f(q^it)^{(-1)^i\binom{k}{i}}.$$

The fourth assertion follows from the matrix factorisation of $A(\Gamma)$; any coefficient $a_{\ii\jj}$ with $[\ii]=[\jj]=J$ can be written as a sum of terms such as 
$$(f_{p\ii-\ii_1}^{(1)}(\Gamma))^{\tau^{m-1}}\cdots f^{(1)}_{p\ii_{m-1}-\jj}(\Gamma).$$
In order for this last term to be non zero, we must have $[\jj]\subseteq[\ii_{m-1}]\subseteq\ldots \subseteq [\ii_{1}]\subseteq [\ii]$ as above. Thus we get $[\ii]=[\ii_1]=\ldots = [\ii_{m-1}]= [\jj]$, and this is what we wanted to show.

Finally, using the definition of the series $F_m$, we see that $f_{\ii}^{(m)}(\Gamma)$ is a sum of terms of the form $\prod_{\dd\in D} \gamma_{\dd}u_{\dd} \lambda_{u_d}^{(m)}$, for some integers $u_{\dd}\geq 0$ such that $\sum_D \dd u_{\dd}=\ii$. If we have $[\ii]=I$, and $D\subset \N_J$ for some $J$ strictly contained in $I$, this last equality is impossible, and we have $f_{\ii}^{(m)}(\Gamma)=0$. From the description of the coefficients of $A(\Gamma)_I$, this ends the proof.
\end{proof}

\section{Minimal solutions, their supports and their digits}
\label{irrsol}

In this section, we first recall some facts and notations about the sets of solutions of certain modular equations, and some of their properties. The ideas come from work of Moreno, Kumar, Castro and Shum \cite{mm}, and have been developed further in \cite{bl}. The reader interested in more details and the proofs can refer to this last paper. The new feature here is the introduction of irreducible solutions, the minimal support, and the description of minimal solutions with given support.

\subsection{The density, and preliminary results} Most of the material presented here comes from \cite{bl}; for this reason, the proofs are omitted when they already appear in this paper.

\begin{definition}
Let $D$ be a finite subset of $\N^n$, and $r$ denote a positive integer. We assume that the set $D$ is not contained in any of the  coordinate hyperplanes of $\R^n$. Recall that $s_p$ denotes the $p$-weight.

For any $\ii\in \Z^n$, the notation $\ii>0$ means that all coordinates of $\ii$ are positive.
 
We define $E_{D,p}(r)$ as the set of $\# D$-tuples $U=(u_{\dd})_{\dd\in D}\in \{0,\dots,p^r-1\}^{\# D}$ that are solutions of
$$\sum_D u_{\dd} \dd \equiv 0~[p^r-1],~\sum_D u_{\dd} \dd>0.$$
For any $U \in E_{D,p}(r)$, we define its {\rm $p$-weight} as the integer $s_p(U)= \sum_{\dd\in D} s_p(u_{\dd})$ and the {\rm length of $U$} as $\ell(U)=r$.
Finally we set $s_{D,p}(r):=\min_{U\in E_{D,p}(r)} s_p(U)$.
\end{definition}  

Moreno et al. \cite{mm} introduce the set $E_{D,p}(r)$ in order to give a lower bound for the $\pi$-adic valuations of exponential sums over $\F_{p^r}$ associated to polynomials with their exponents in $D$ and coefficients in this field. Actually they show that a lower bound for these valuations is $s_{D,p}(r)$, and that this bound is attained.

In order to study the valuation of the coefficients of the $L$-function, we have to make $r$ vary; in \cite[Proposition 1.1]{bl}, we proved the following 

\begin{proposition}
\label{princ}
The set $\left\{\frac{s_{D,p}(r)}{r}\right\}_{r\geq 1}$ has a minimum.
\end{proposition}

This result allows the definition of $p$-density; this invariant is particularly important here. It is a sharp lower bound for the valuations of the reciprocal roots and poles of the $L$-functions \cite[Theorem 3.2]{bl}.

\begin{definition}
Let $D$ be a finite subset of $\N^n$, and $p$ be a prime.
\begin{itemize}
	\item[(i)] Assume that $D$ is not contained in any of the  coordinate hyperplanes of $\R^n$. The {\rm $p$-density of the set $D$} is the rational number
$$\delta_p(D):= \frac{1}{p-1}\min_{r\geq 1}\left\{\frac{s_{D,p}(r)}{r}\right\}.$$
	\item[(ii)] Assume that $D$ is contained in some of the  coordinate hyperplanes of $\R^n$; we set $\delta_p(D):=\infty$
\item[(iii)] The density of a solution $U\in E_{D,p}(r)$ is $\delta(U):=\frac{s_p(U)}{(p-1)r}$. The element $U$ is {\rm minimal} when $\delta(U)=\delta_p(D)$.
\end{itemize}
\end{definition}

We shall need some auxiliary results in the following

\begin{lemma}
\label{qmin}
Let $(u_{\dd})_D$ be nonnegative integers such that $\sum_D u_{\dd} \dd \equiv 0~[p^r-1]$ and the sum $\sum_D u_{\dd} \dd$ has all its coordinates positive. Then we have the inequality $\sum_D s_p(u_{\dd})\geq s_{D,p}(r)$.
\end{lemma}

\begin{proof}
For $u$ a positive integer, let $\overline{u}\in \{1,\cdots,p^r-1\}$ be the integer defined by $u\equiv \overline{u}$ mod $p^r-1$. If $u=0$, we set $\overline{u}:=0$. Then the congruence $\sum_D \overline{u}_{\dd} \dd \equiv 0~[p^r-1]$ trivially holds, and the sum $\sum_D \overline{u}_{\dd} \dd$ has all its coordinates positive. Finally the inequality $\sum_D s_p(\overline{u}_{\dd})\geq s_{D,p}(r)$ comes from the definition of this last integer, and the result is a consequence of the following lemma.
\end{proof}

\begin{lemma}
\label{qred}
Notations being as in the proof above, we have the inequality $s_p(u)\geq s_p(\overline{u})$. 
\end{lemma}

\begin{proof}
Write the euclidean division of $u$ by $p^r$, $u=p^ru_1+v_1$. The $p$-weights of these integers satisfy $s_p(u)=s_p(u_1)+s_p(v_1)\geq s_p(u_1+v_1)$. Replacing $u$ by $u_1+v_1$, and repeating the same process, we finally get $\overline{u}$ and the result.
\end{proof}

We end these preliminary results giving an inequality between the densities of the set $D$ and of the subsets $D_I$ from Definition \ref{defNI}. It explains the range of the product in Theorem \ref{Lcong}.

\begin{lemma}
\label{DetDI}
Assume that $D$ is not contained in any of the  coordinate hyperplanes of $\R^n$. Let $I\subset\{1,\ldots,n\}$, with $\# I=k$; then we have the inequality $\delta_p(D)\leq \delta_p(D_I)+n-k$.
\end{lemma}

\begin{proof}
If $D_I$ is contained in some coordinate hyperplane of $\N_I$, then we have $\delta_p(D_I)=\infty$, and there is nothing to prove. Else let $U=(u_{\dd})_{D_I}$ denote a minimal solution in $E_{D_I,p}(r)$. From our hypothesis on $D$, we can choose some $\dd_1\in D$ whose support $[\dd_1]$ is not contained in $I$; if $I\cup[\dd_1]=\{1,\ldots,n\}$, we stop; else we choose some $\dd_2$ such that $[\dd_2]\subsetneq I\cup[\dd_1]$, until we get $\dd_1,\dd_2,\ldots,\dd_l$ with $I\cup[\dd_1]\cup\ldots\cup[\dd_l]=\{1,\ldots,n\}$. We must have $l\leq n-k$.

Now consider $V:=((u_{\dd})_{D_I},u_{\dd_1}=p^r-1,\ldots,u_{\dd_l}=p^r-1)$; this is an element in $E_{D,p}(r)$, with density $\delta(V)=\delta(U)+l\leq \delta_p(D_I)+n-k$. This ends the proof of the lemma.
\end{proof}

\subsection{Minimal solutions and their supports}

In this section, we assume that $D$ is not contained in any of the  coordinate hyperplanes of $\R^n$. We focus on minimal elements in the sets $E_{D,p}(r)$ and we define their supports. They appear in Section 3 in the location of the minors with minimal valuation for the matrices of Dwork's operators.

\begin{definition}
Let $\delta_r$ be the {\rm shift}, the map from the set $\{0,\dots,p^r-1\}$ to itself, which sends any integer $0\leq k\leq p^r-2$ to the residue of $pk$ modulo $p^r-1$, and $p^r-1$ to itself. 
We define a map
$$\begin{array}{ccccc}
\varphi_r & : & E_{D,p}(r) & \rightarrow & (\N_{>0})^n \\
        &   & U      & \mapsto & \frac{1}{p^r-1} \sum_{\dd\in D} \dd u_{\dd}\\
        \end{array}$$
To each solution $U$ in $E_{D,p}(r)$, we associate a map $\varphi_U$ from $\Z/r\Z$ to $\N_{>0}^n$ defined by 
$$\varphi_U(k):=\varphi_r(\delta_r^k(U));$$
we say $U$ is {\rm irreducible} when $\varphi_U$ is an injection. We call $\varphi_U$ the {\rm support} of $U$.

We denote by $MI_{D,p}(r)$ the set of minimal irreducible elements in $E_{D,p}(r)$.
\end{definition}

\begin{remark}

Note that the map $\delta_r$ acts on $E_{D,p}(r)$; moreover it shifts the $p$-digits, hence his name. As a consequence, it preserves the $p$-weight, and minimality. 

Note also that minimal irreducible elements do exist for some $r$, as shows the reduction process in the proof of \cite[Proposition 1.1]{bl}.
\end{remark}

We begin with a lemma

\begin{lemma} 
\label{mini}
We have the following
\begin{itemize}
	\item[(i)] The sets $MI_{D,p}(r)$ are pairwise disjoint;
	\item[(ii)] the set $MI_{D,p}(r)$ is empty for $r$ large enough;
	\item[(iii)] the map $\delta_r$ sends $MI_{D,p}(r)$ to itself; moreover we have for any $i,k$ $\varphi_{\delta_r^k U}(i)=\varphi_U(i+k)$.
\end{itemize}

\end{lemma}

\begin{proof}
Note that an element $U\in \N^{\# D}$ can belong simultaneously to various sets $E_{D,p}(r)$ when $r$ varies, but at most to one set $MI_{D,p}(r)$ since the density of $U\in E_{D,p}(r)$ depends on the integer $r$. This proves (i), while assertion (ii) is a consequence of \cite[Lemma 1.3 iii/]{bl}. 

Concerning the last assertion, it suffices to remark that the map $\delta_r$ preserves the $p$-weight (hence minimality); the assertion about the support is a direct consequence of the definitions.

\end{proof}

We define some sets coming from minimal irreducible solutions.

\begin{definition}

We set 
$$MI_{D,p}:=\coprod_r MI_{D,p}(r):=\{U_1,\ldots,U_t\}.$$ 
We define the $p$-{\rm minimal support} of $D$ as
$$\Sigma_{p}(D)=\cup_{i=1}^t \Ima \varphi_{U_i},$$
and set $N_{p}(D):=\#\Sigma_{p}(D)$ to denote its cardinality.
\end{definition}

Note that the sets $MI_{D,p}(n)$, $MI_{D,p}$ and $\Sigma_p(D)$ are finite from Lemma \ref{mini}.

\medskip

We prove a lemma that we shall use further: it will help us locate the minors with minimal valuation in a matrix of the operator $\alpha_m$ in Section 3. 

\begin{lemma} 
\label{suppmin}
Let $U\in E_{D,p}(r)$ be a minimal solution. Then the image $\Ima \varphi_U$ of its support is contained in the minimal support $\Sigma_{p}(D)$.
\end{lemma}

\begin{proof}
If $U$ is irreducible, this is clear from the definition of the minimal support $\Sigma_{p}(D)$. Else we can find integers $t_1< t_2$ in $\{0,\dots,r-1\}$ such that $\varphi(\delta_r^{t_1}(U))=\varphi(\delta_r^{t_2}(U))$. If we consider the element $\delta_r^{t_1}(U)$ instead of $U$ (they have the same $p$-weight, and their supports are shifted), we obtain some $0<t\leq r-1$ such that $\varphi(U)=\varphi(\delta_r^{t}(U))$.

For each $\dd$, let $u_{\dd}=p^{r-t}w_{\dd}+v_{\dd}$ be the result of the euclidean division of $u_{\dd}$ by $p^{r-t}$. Define the $\#D$-uples $V=(v_{\dd})$, and $W=(w_{\dd})$. From \cite[Lemma 1.2 ii/]{bl} and the definition of $t$, we have
$$\sum_{D} \dd v_{\dd}=(p^{r-t}-1)\varphi(U)~;~\sum_{D} \dd w_{\dd}=(p^{t}-1)\varphi(U).$$
Thus the tuples $V$ and $W$ are respectively contained in $E_{D,p}(r-t)$ and $E_{D,p}(t)$. From the definition of $p$-density, both $\delta(V)$ and $\delta(W)$ are greater than or equal to $\delta_p(D)$. But for each $\dd$ we have $s_p(u_{\dd})=s_p(v_{\dd})+s_p(w_{\dd})$, and $s_p(U)=s_p(V)+s_p(W)$. Since $U$ is minimal, we have 
$$\delta_p(D)=\frac{s_p(V)+s_p(W)}{r(p-1)}=\left(1-\frac{t}{r}\right)\delta(V)+\frac{t}{r}\delta(W)\geq \delta_p(D).$$ 
Thus both solutions $V$ and $W$ are minimal, and again from \cite[Lemma 1.2 ii/]{bl}, we have $\Ima \varphi_U=\Ima \varphi_V\cup\Ima \varphi_W$. If both $V$ and $W$ are irreducible, we are done; else we apply the same process to $V$ or $W$, and we end with minimal irreducible elements $U_{1},\cdots,U_{k}$ in $E_{D,p}(r_1),\cdots , E_{D,p}(r_k)$ with $\Ima \varphi_U=\cup_i \Ima \varphi_{U_{i}}$, each support $\Ima \varphi_{U_{i}}$ being contained in the minimal support from its definition.
\end{proof}

We conclude this subsection by considering how one can glue together two minimal solutions.

\begin{lemma}
\label{gluemin}
Let $U\in E_{D,p}(r)$, and $U'\in E_{D,p}(r')$ denote two minimal solutions, such that $\varphi_U(0)=\varphi_{U'}(0)$. We define $V=(v_{\dd})_D$ by $v_{\dd}:=p^ru_{\dd}'+u_{\dd}$.

Then $V$ is a minimal element in $E_{D,p}(r+r')$, whose support is $\varphi_V$ defined by $\varphi_V(i)=\varphi_{U'}(i)$ for $0\leq i\leq r'$, and $\varphi_V(i+r')=\varphi_{U}(i)$ for $1\leq i\leq r-1$.
\end{lemma}

\begin{proof}
An easy calculation gives 
$$\sum_D \dd v_{\dd}=p^r\sum_D \dd u'_{\dd}+\sum_D \dd u_{\dd}=p^r(p^{r'}-1)\varphi_{U'}(0)+(p^r-1)\varphi_{U}(0)=(p^{r+r'}-1)\varphi_{U}(0),$$
and $V$ is a solution of length $r+r'$, with $\varphi_{V}(0)=\varphi_{U}(0)$; its density is a barycenter (as in the proof of Lemma \ref{suppmin}) of the densities of $U$ and $U'$. Thus it is also minimal.
We show the assertion about the support. Fix some $1\leq i\leq r$; for any $\dd\in D$, the remainder of the euclidean division of $v_{\dd}$ by $p^i$ is the same as that of $u_{\dd}$ by $p^i$; from \cite[Lemma 1.2 ii/]{bl} and since we have $\varphi_{V}(0)=\varphi_{U}(0)$, we get $\varphi_{V}(-i)=\varphi_{U}(-i)$. Now fix $r+1\leq i\leq r+r'-1$. Here the remainder of the euclidean division of $v_{\dd}$ by $p^i$ is $p^r w_{\dd}+u_{\dd}$, where $w_{\dd}$ is the remainder of the euclidean division of $U$ by $p^{i-r}$. Using \cite[Lemma 1.2 ii/]{bl} again, we get $\varphi_{V}(-i)=\varphi_{U'}(r-i)$. This is the result.

\end{proof}

\subsection{Digits of minimal solutions}
We end this section with some results about base $p$ digits of minimal solutions. We also define some sets that will be the building blocks for the matrices $M(\Gamma)$ appearing in the congruence.

\begin{definition}
\label{basepdigit}
Define the map $\psi:MI_{D,p}\rightarrow \{0,\ldots,p-1\}^{\# D}$ by $\psi(U)=(u_{d0})_D$, where for each $d\in D$, $u_{d0}$ is the remainder of the euclidean division of $u_d$ by $p$.
Let $\ee,\ee'\in \Sigma_p(D)$ denote elements in the minimal support. We define the set $V(\ee,\ee')$ as
$$V(\ee,\ee'):=\{\psi(U),~U\in MI_{D,p},~\varphi_U(-1)=\ee,~\varphi_U(0)=\ee'\}\subset \{0,\ldots,p-1\}^{\# D}.$$
For any $V=(v_{\dd})_D\in \{0,\ldots,p-1\}^{\# D}$, we define its {\rm weight} by $w(V):=\sum_D v_{\dd}$.
\end{definition}

\begin{remark}
Note that the conditions $\varphi_U(-1)=\ee,~\varphi_U(0)=\ee'$, joint with \cite[Lemma 1.2 ii/]{bl}, ensure the equality $\sum_D \dd v_{\dd}=p\ee-\ee'$. 
\end{remark}

We begin with some technical results

\begin{lemma}
\label{constantweight}
Let $\ee,\ee'\in \Sigma_p(D)$, and assume $V(\ee,\ee')$ is non empty. Then all $v\in V(\ee,\ee')$ have the same weight. We denote it by $w(\ee,\ee')$.
\end{lemma}

\begin{proof}
Choose $V,V'\in V(\ee,\ee')$, and assume $w(V)<w(V')$. We can find some $U\in MI_{D,p}(r)$, $U'\in MI_{D,p}(r')$ with $\psi(U)=V$ and $\psi(U')=V'$. Define $U''$ by setting $u_{\dd}''=u_{\dd}'-v_{\dd}'+v_{\dd}$ for any $\dd\in D$; we get $U''\in E_{D,p}(r')$ since $\sum_D \dd v_{\dd}=\sum_D \dd v_{\dd}'=p\ee-\ee'$. Moreover we have $s_p(U'')=s_p(U')-w(V')+w(V)<s_p(U')$, contradicting the minimality of $U'$. Thus all elements in $V(\ee,\ee')$ have the same weight.
\end{proof}

\begin{lemma}
\label{basepglue1}
Fix some (not necessarily distinct) $\ee_{-1},\ee_0,\ldots,\ee_{k-1} \in \Sigma_{D,p}$, such that for any $0\leq i\leq k-1$ we have $V(\ee_i,\ee_{i-1})\neq \emptyset$. 

For any $(V_i)_{0\leq i\leq k-1}\in \prod_{i=0}^{k-1} V(\ee_i,\ee_{i-1})$, there exists some minimal $U=(u_{\dd})_D$ of length $r\geq k$ such that the remainder of the euclidean division of $u_{\dd}$ by $p^k$ is $\sum_{i=0}^{k-1} p^i v_{i\dd}$. Moreover its support verifies $\varphi_U(i)=\ee_{r-1-i}$ for any $r-k\leq i\leq r$.
\end{lemma}

\begin{proof}
We show this result by induction on $k$. For $k=1$, this comes directly from the definitions of the map $\psi$ and the set $V(\ee_0,\ee_{-1})$: any $U\in MI_{D,p}(r)$ such that $\psi(U)=V$ satisfies the requirements of the lemma.

Assume the result is true for $k$ elements in the minimal support. Choose elements $\ee_{-1},\ee_0,\ldots,\ee_{k} \in \Sigma_{D,p}$ and $(V_i)_{0\leq i\leq k}$ as above. From the induction hypothesis, we get $U$ minimal of length $r\geq k$ with $\varphi_U(i)=\ee_{r-1-i}$ for any $r-k\leq i\leq r$; then from Lemma \ref{mini} we have $\varphi_{\delta_r^{r-k}U}(0)=\varphi_{U}(r-k)=\ee_{k-1}$. On the other hand we choose a minimal $U'$ of length $r'$ such that $\psi(U')=V_k$; we have $\varphi_{U'}(0)=\ee_{k-1}$. From Lemma \ref{gluemin}, $U''$ defined for any $\dd\in D$ by $u_{\dd}'':=p^ru_{\dd}'+\delta^{r-k}u_{\dd}$ is minimal of length $r+r'\geq k+1$. Moreover from this construction, we have the following base $p$ expansions for the $u_{\dd}''$
$$u_{\dd}''=\sum_{j=r+1}^{r+r'-1} p^j\ast+p^rv_{k\dd}+p^{r-k}\sum_{i=0}^{k-1}p^iv_{i\dd}+\sum_{j=0}^{r-k-1} p^j\ast$$
for some integers $\ast$ in $\{0,\ldots,p-1\}$. Thus $W:=\delta^{k-r} U''$ satisfies the requirement about the remainder. The one about the supports is a consequence of the last assertion of lemma \ref{gluemin}: the support of $U''$ satisfies $\varphi_{U''}(r'-1+i)=\ee_{k-i}$ for any $0\leq i\leq k+1$, thus the support of $W=\delta^{k+r'} U''$ satisfies $\varphi_W(i)=\varphi_{U''}(k+r'+i)=\ee_{r+r'-i-1}$ for any $r+r'-k-1\leq i\leq r+r'$; this is the last claim.
\end{proof}

\begin{corollary}
\label{basepglue}
Choose $\ee_{-1},\ldots,\ee_{n-1} \in \Sigma_{D,p}$ (not necessarily distinct), with $\ee_{-1}=\ee_{n-1}$, and such that for any $0\leq i\leq n-1$ we have $V(\ee_i,\ee_{i-1})\neq \emptyset$. 

For any $0\leq i\leq n-1$, choose $V_i=(v_{i\dd})_D\in V(\ee_i,\ee_{i-1})$; then $U$ defined by $u_{\dd}:=\sum_{i=0}^{n-1}p^i v_{i\dd}$ is a minimal solution in $E_{D,p}(n)$, with support defined for any $0\leq i\leq n-1$ by $\varphi_U(i)=\ee_{n-1-i}$
\end{corollary}

\begin{proof}
From the lemma above, one can construct $U$, minimal of length $r\geq n$, such that the remainder of the euclidean division of $u_{\dd}$ by $p^n$ is $u_{\dd}'=\sum_{i=0}^{n-1} p^i v_{i\dd}$. Moreover its support verifies $\varphi_U(i)=\ee_{r-1-i}$ for any $r-n\leq i\leq r$. If we apply \cite[Lemma 1.2 ii/]{bl}, we get $\sum_D \dd u_{\dd}'=p^n\varphi_U(-n)-\varphi_U(0)=(p^n-1)\ee_{-1}$. Thus $U'=(u_{\dd}')_D$ is an element of $E_{D,p}(n)$, with $\varphi_{U'}(0)=\ee_{-1}$. As in the proof of Lemma \ref{suppmin}, it is minimal since it comes from a minimal $U$. Finally the assertion about its support comes from \cite[Lemma 1.2 ii/]{bl} applied to the reductions modulo $p^i$ of $U$ and $U'$, which are the same for any $0\leq i\leq n-1$.
\end{proof}

With this in hand we give a decomposition of minimal solutions with fixed supports in term of their base $p$ digits.

\begin{proposition}
\label{decbasep}
Let $m\geq 1$ be an integer, and $\varphi$ denote a map from $\Z/m\Z$ to $\Sigma_p(D)$. If $M_{D,p}(\varphi)$ denotes the set of minimal elements in $E_{D,p}(m)$ whose support is $\varphi$, then we have

\begin{itemize}
	\item[(i)] the set $M_{D,p}(\varphi)$ is empty if, and only if (at least) one of the sets $V(\varphi(-i-1),\varphi(-i))$ is empty;
	\item[(ii)] else the map
$$B_\varphi : M_{D,p}(\varphi) \rightarrow \prod_{i=0}^{m-1} V(\varphi(-i-1),\varphi(-i))$$
sending $(u_{\dd})_{\d\in D}$ to its base $p$ digits $(u_{\dd,i})_{\d\in D, 0\leq i\leq m-1}$ is one-to-one.
\end{itemize}
\end{proposition}

\begin{proof}
First assume every set $V(\varphi(-i-1),\varphi(-i))$ is non empty; from the corollary above, we construct a minimal solution with support $\varphi$. Conversely, consider a minimal solution with support $\varphi$; from the decomposition of a minimal element in terms of irreducible ones, as in the proof of Lemma \ref{suppmin}, the map $B_\varphi$ is well defined, and we get an element in the set $\prod_{i=0}^{m-1} V(\varphi(-i-1),\varphi(-i))$ from \cite[Lemma 1.2 ii/]{bl}; thus none of the sets $V(\varphi(-i-1),\varphi(-i))$ is empty; this shows assertion (i).

To show assertion (ii), we just have to remark that the map from $\prod_{i=0}^{m-1} V(\varphi(-i-1),\varphi(-i))$ to $M_{D,p}(\varphi)$ sending $(v_{\dd,i})_{\d\in D, 0\leq i\leq m-1}$ to $(\sum_{i=0}^{m-1} p^i v_{\dd,i})_{\d\in D}$ is well defined from the corollary above. Moreover it is the inverse function of the map $B_\varphi$.
\end{proof}

\section{Congruences}
\label{secmain}

We fix a subset $D\subset \N^n$, not contained in any coordinate hyperplane, and a prime $p$. In the following $f$ is a polynomial having its exponents in $D$, to which we associate the series $F_1(\Gamma,\X)$ and $F_m(\Gamma,\X)$, and the matrices $A(\Gamma)_I$, $B(\Gamma)_I$ defined in the first section.

The aim of this section is to give a congruence for the Fredholm determinants of the matrices $A(\Gamma)_I$ defined above, and to deduce similar results for the $L$-functions. In order to do this, we examine the minors of the matrices $A_I$.

In the first two subsections, we consider the matrix $A(\Gamma)_{\{1,\ldots,n\}}$. In the first, we show that the entire function $\det(\I-TA(\Gamma)_I)$ lies in the ring $M_{\delta_p(D_I)}$. From this result, we consider this function modulo $I_\delta$ along the second subsection. We are lead to the study of the minimal solutions defined in the preceding section. From the factorization in Lemma \ref{exp} iv/, we have to consider the matrix $B_{\{1,\ldots,n\}}$; we show that the indices of the lines (and columns) of a ``minimal'' minor must lie in the $p$-minimal support of $D$ defined in the preceding section. Then we are able to give a congruence for Fredholm determinants, in Proposition \ref{congmanin}. 

At the end of the section, we use these results to prove Theorem \ref{Lcong}.

\subsection{The Fredholm determinant, and the minimal support}

In this subsection and the next one, we set $A:=A(\Gamma)_{\{1,\ldots,n\}}$ and $B:=B(\Gamma)_{\{1,\ldots,n\}}$ in the notations of Definition \ref{defNI}; we also drop the $\Gamma$ when no confusion can occur. Write $\det(\I-TA):=1+\sum_{s\geq 1} \ell^{(m)}_s T^s$, and $\det(\I-TB):=1+\sum_{s\geq 1} \ell^{(1)}_s T^s$. We also denote by $\delta:=\delta_p(D)$, $\Sigma:=\Sigma_{p}(D)$ and $N:=N_{p}(D)$ the $p$-density of $D$, $p$-minimal support of $D$ and the cardinality of this last set all along this section. Moreover we set $\Sigma:=\{\ee_1,\cdots,\ee_N\}$.

\medskip

Our aim is to get a congruence for the Fredholm determinant $\det(\I-TA)$. We begin by recalling some facts about the coefficients $\ell_s^{(m)}$: we shall decompose their principal parts as sums of terms that we link to the minimal solutions defined in the preceding section. The calculations are rather tedious, but the idea is simple: the coefficients $\ell_s^{(m)}$ are expressed in terms of the coefficients of the functions $F_m$ (and $F_1$ from the factorisation of $A$). In turn, these last coefficients can be written from the coefficients $\gamma_{\dd}$ of the lifting of the polynomial $f$, and the coefficients $\lambda_i^{(\cdot)}$ of Dwork's splitting functions $\theta_1$ and $\theta_m$. Finally, the valuations of the coefficients $\ell_s^{(m)}$ come from those of the $\lambda_i^{(\cdot)}$ (which are greater than or equal to $\sigma_p(i)$ from Lemma \ref{coeffssplit}), and a careful examination of their expression leads us to minimal elements in $E_{D,p}(mn)$ for some $n\leq s$.

Let $F$ be a non-empty subset of $(\N_{>0})^n$. We denote by $A[F]$ the matrix $(f^{(m)}_{q\ii-\jj})_{\ii,\jj\in F}$. We have the following expression for the coefficient $\ell_s^{(m)}$ in terms of the determinants of the matrices $A[F]$
$$\ell_s^{(m)}=\sum_{F\subset (\ma{N}_{>0})^n,~\# F=s} \det A[F].$$
From the definition of the determinant, we have, for $F=\{\uu_0,\dots,\uu_{s-1}\}$ (where as usual $S_s$ denotes the symmetric group over $s$ elements)

$$\det A[F]=\sum_{\sigma\in S_s} M_{F,\sigma},\quad M_{F,\sigma}:=\sgn(\sigma)\prod_{i=0}^{s-1} f^{(m)}_{q\uu_i- \uu_{\sigma(i)}},$$

We present another, less classical, expression for the determinant. It comes from the decomposition of permutations as products of disjoint cycles. Let us give some definitions in order to introduce it.

\begin{definition}
\label{isigma}
We denote by $\II_l$ the set of injections from $\Z/l\Z$ to $(\N_{>0})^n$. For $\theta\in \II_l$, we define the \emph{cyclic minor} associated to $\theta$ as $$M_\theta^{(m)}:=(-1)^{l-1}\prod_{i=0}^{l-1} f^{(m)}_{q\theta(i)-\theta(i+1)}.$$

Let $F$ be a non-empty subset of $(\N_{>0})^n$, with cardinality $s$. Define the finite set $\II(F):=\coprod_{k=1}^s \II_k(F)$, where $\II_k(F)$ is the set of injections from $\Z/k\Z$ to $F$. 

Let $\AA(F)$ consist of the subsets $\Theta:=\{\theta_1,\dots,\theta_{|\Theta|}\} \subset \II(F)$ such that 
\begin{itemize}
	\item[(i)] for each $i$, $\theta_i(0)=\min \Ima \theta_i$; 
	\item[(ii)] the $\Ima \theta_i$, $1\leq i\leq |\Theta|$, form a partition of $F$.
\end{itemize}
\end{definition}

From this new set, we can rewrite the determinant $\det A[F]$ in terms of cyclic minors

\begin{lemma}
\label{decdetcyc}
Notations being as in the definition above, we have the following expression for the determinant $\det A[F]$

$$\det A[F] = \sum_{\Theta \in \AA(F)} \prod_{i=1}^{|\Theta|} M_{\theta_i}^{(m)}.$$
\end{lemma}

\begin{proof}
First recall that any $\sigma \in S_s$ can be written in a unique way, up to permutation, as $\sigma=\gamma_1\cdots \gamma_{|\sigma|}$ where the $\gamma_i$ are cycles whose supports form a partition of $\{0,\cdots,s-1\}$. Such a cycle $\gamma_i$ of length $s_i$ can be represented in a unique way as an injection $\eta_i$ from $\Z/s_i\Z$ to $\{0,\cdots,s-1\}$ such that $\eta_i(0)=\min \Ima \eta_i$ as $\gamma_i=(\eta_i(0)\cdots \eta_i(s_i-1))$. Thus the map $\sigma \mapsto \{\eta_1,\cdots,\eta_{|\sigma|}\}$ defines a bijection between the sets $S_s$ and $\AA(\{0,\cdots,s-1\})$.

Let $g$ be the bijection from $\{0,\cdots,s-1\}$ to $F$ sending $i$ to $\uu_i$; from what we have just said, the map $\sigma \mapsto \{g\circ \eta_1,\cdots,g \circ \eta_{|\sigma|}\}$ is a bijection from $S_s$ to $\AA(F)$. Now for any $\sigma\in S_s$ with image $\{\theta_1,\cdots,\theta_{|\sigma|}\}$ in $\AA(F)$, we have $M_{F,\sigma}=\prod_{i=1}^{|\sigma|} M_{\theta_i}^{(m)}$. This is the desired result.
\end{proof}

In order to give congruences for the cyclic minors, recall that we have the following expansion for the coefficients of the series $F_m(\Gamma,\X)$
\begin{equation}
\label{coeffm}
f_{\ii}^{(m)}(\Gamma)=\sum_{\sum \dd u_{\dd}=\ii} \prod_D \lambda_{u_{\dd}}^{(m)} \gamma_{\dd}^{u_{\dd}}
\end{equation}

We first show two important facts. On one hand, the valuations of the minors are bounded below by a linear function of their size, the coefficient being the density. On the other hand, when we look at the minors whose valuation attains this bound, we can restrict our attention to the cyclic minors whose support is contained in the $p$-minimal support of $D$.

\begin{lemma}
\label{supp1}
Choose some injection $\theta$ in $\II_l$. We have the following congruences in the ring $\Z_p[\zeta_p]$
\begin{itemize}
	\item[(i)] $M_\theta^{(m)}\equiv 0 \mod \pi^{ml(p-1)\delta}$, and
	\item[(ii)] $M_\theta^{(m)}\equiv 0 \mod \pi^{ml(p-1)\delta+1}$ when the image of $\theta$ is not contained in the $p$-minimal support of $D$.
\end{itemize}
\end{lemma}

\begin{proof}
From the expression of $M_\theta^{(m)}$ and (\ref{coeffm}), we can write $M_\theta^{(m)}$ as a sum of terms of the form
$$A_{(u_{\dd}^{i})}= (-1)^{l-1}\prod_{i=0}^{l-1} \prod_D \lambda_{u^{i}_{\dd}}^{(m)} \gamma_d^{u^{i}_{\dd}},$$

where the $(u_{\dd}^i)_{\dd\in D,~0\leq i\leq l-1}$ satisfy $\sum_D du_{\dd}^{i}=q\theta(i)-\theta(i+1)$ for each $i$. From Lemma \ref{coeffssplit}, the valuation of such a term satisfies

\begin{equation}
\label{valq}
v(A_{(u_{\dd}^{i})})\geq \sum_{i=0}^{l-1} \sum_D s_p(u_{\dd}^{i}),
\end{equation}
with equality if and only if we have $0\leq u_{\dd}^{i}\leq q-1$ for any $i,\dd$.

For each $\dd$ in $D$, set $u_{\dd}=\sum_{i=0}^{l-1} q^{l-1-i} u_{\dd}^{i}$. A rapid calculation gives the equality
\begin{equation}
\label{equa}
\sum_D \dd u_{\dd}=(q^l-1)\theta(0).
\end{equation}
By the way we defined the integers $u_{\dd}$, we have the inequality $\sum_{i,\dd} s_p(u_{\dd}^{(i)})\geq \sum_D s_p(u_{\dd})$, and from Lemma \ref{qmin}, the inequality $\sum_D s_p(u_{\dd})\geq ml(p-1)\delta$. Together with equation (\ref{valq}), this proves assertion (i).

Assume that $v(A_{(u_{\dd}^{i})})=ml(p-1)\delta$. Then the three inequalities above are equalities. But equality in (\ref{valq}) implies the second equality, and that $0\leq u_{\dd}\leq p^{ml}-1$ for each $\dd$. Thus $U:=(u_{\dd})$ is a solution in $E_{D,p}(ml)$ from (\ref{equa}). Equality for the third gives that $U$ is a minimal solution in $E_{D,p}(ml)$. Finally, from the definition of $u_{\dd}$ and \cite[Lemma 1.2 ii/]{bl}, we have for any $i$ that
$$\varphi_U(mi)=\theta(i),~i \in \Z/l\Z,$$
and $\Ima \theta \subset \Ima \varphi_U$. The second assertion now follows from Lemma \ref{suppmin}.
\end{proof}

We give a first congruence for the Fredholm determinant.
  
\begin{lemma}
\label{supp2}
Let $F\subset (\N_{>0})^n$, with cardinality $s$. Then we have the congruences
\begin{itemize}
  \item[(i)] for any $s\geq 1$, we have $\ell_s^{(m)}\equiv 0 \mod \pi^{ms(p-1)\delta}$; as a consequence, the series $\det(\I-TA)$ is in $M_\delta$;
	\item[(ii)] $\det A[F]\equiv 0 \mod \pi^{ms(p-1)\delta+1}$ if $F$ is not contained in the minimal support $\Sigma$;
	\item[(iii)] in the ring $M_\delta$, we have
	$$\det(\I-T A)\equiv \det(\I_N-T A[\Sigma]) \mod I_\delta$$.
\end{itemize}
\end{lemma}

\begin{proof}
Recall from Lemma \ref{decdetcyc} that we can write $\det A[F]$ from cyclic minors, as a sums of terms $M_{\theta_1}^{(m)}\cdots M_{\theta_k}^{(m)}$, with the $\Ima \theta_i$ pairwise disjoint and $\coprod\Ima \theta_i =F$.

From Lemma \ref{supp1} (i), we get the inequality $v(M_{\theta_i}^{(m)}) \geq  ms_i(p-1)\delta$, where we have set $s_i:=\# \Ima \theta_i$. We get assertion (i) since $s=s_1+\cdots+s_k$, and the coefficient $\ell_s^{(m)}$ is the sum of the $\det A[F]$ when $F$ describes the subsets of $(\N_{>0})^n$ with cardinality $s$.

If $F\nsubseteq \Sigma$, there exists at least one $\theta_i$ such that $\Ima \theta_i\nsubseteq \Sigma$, and from Lemma \ref{supp1} (ii), we get the inequality $v(M_{\theta_i}^{(m)}) \geq  ms_i(p-1)\delta+1$; reasoning as above, we get assertion (ii).

As a consequence, in the ring $M_\delta$, the only terms in the Fredholm determinant remaining after reduction modulo $I_\delta$ are those coming from principal minors whose support is contained in $\Sigma$. This gives the congruence in the power series ring $M_\delta$.
\end{proof}

\subsection{The congruence for the Fredholm determinants}

We shall now use the factorisation of the matrix $A$ in terms of $B$: recall from Lemma \ref{exp} (iv) that we have $A=B^{\tau^{m-1}}\cdots B$, and that $B$ is the matrix $(f^{(1)}_{p\ii-\jj})_{\ii,\jj>0}$. As above we begin by considering cyclic minors. From Lemma \ref{supp1}, we choose some injection $\theta\in \II_l$ whose image is contained in $\Sigma$, and let $M_\theta^{(m)}$ be as above. From \cite[Lemma 3.2]{bf} and the factorization above we can write

$$M_\theta^{(m)}=(-1)^{l-1}\sum_{(\theta_1,\cdots,\theta_{m-1})\in \II_l^{m-1}} \prod_{i=0}^{l-1} \prod_{j=0}^{m-1} \left(f_{p\theta_j(i)-\theta_{j+1}(i)}^{(1)}\right)^{\tau^{m-1-j}}.$$

where for each $i$, we have set $\theta_0(i):=\theta(i)$ and $\theta_m(i):=\theta(i+1)$.

The following result is similar to Lemma \ref{supp1}; it ensures that in the expression above, we can restrict our attention to the terms such that all injections $\theta_i$ have their image in the $p$-minimal support of $D$. 

\begin{lemma}
\label{supp2bis}
Notations are as above. Assume that for some $j$ we have $\Ima \theta_j \nsubseteq \Sigma$; then we have the congruence
$$\prod_{i=0}^{l-1} \prod_{j=0}^{m-1} \left(f_{p\theta_j(i)-\theta_{j+1}(i)}^{(1)}\right)^{\tau^{m-1-j}} \equiv 0 \mod \pi^{ml(p-1)\delta+1}.$$
\end{lemma}

\begin{proof}
We have the following expression for the coefficients of the series $F_1$

$$f_{\nn}^{(1)}=\sum_{\sum {\dd}u_{\dd}=\nn} \prod_D \lambda_{u_{\dd}}^{(1)} \gamma_{\dd}^{u_{\dd}}.$$

 Thus $M_\theta^{(m)}$ can be written as a sum of terms of the form

$$A_{(u_{\dd}^{ij})}= \prod_{i=0}^{l-1} \prod_{j=0}^{m-1}\left(\prod_D \lambda_{u^{ij}_{\dd}}^{(1)} \gamma_{\dd}^{u^{ij}_{\dd}}\right)^{\tau^{m-1-j}},$$

where for each $i,j$ we have $\sum_D {\dd}u_{\dd}^{ij}=p\theta_j(i)-\theta_{j+1}(i)$. From Lemma \ref{coeffssplit}, the valuation of such a term satisfies

$$v(A_{(u_{\dd}^{ij})})\geq \sum_{i=0}^{l-1} \sum_{j=0}^{m-1}\sum_D s_p(u_{\dd}^{ij}),$$
with equality if and only if we have $0\leq u_{\dd}^{ij}\leq p-1$ for any $i,j,{\dd}$. 

For each ${\dd}$, set $u_{\dd}=\sum_{i=0}^{l-1} q^{l-1-i} \sum_{j=0}^{m-1} p^{m-1-j} u_{\dd}^{ij}$. As in the proof of Lemma \ref{supp1}, we have
$$\sum_D {\dd}u_{\dd}=(q^l-1)\theta(0).$$
Assume the equality $v(A_{(u_{\dd}^{ij})})=ml(p-1)\delta$. As above, it implies that $U:=(u_{\dd})$ is a minimal solution in $E_{D,p}(ml)$. From \cite[Lemma 1.2 ii/]{bl} we have $\theta_j(i)=\varphi_U(mi+j)$, and Lemma \ref{suppmin} ensures that $\Ima \theta_j$ is contained in the minimal support $\Sigma$. 

As a consequence, if for some $0\leq j\leq m-1$, we have $\Ima \theta_j \nsubseteq \Sigma$, then 
$$\prod_{i=0}^{l-1} \prod_{j=0}^{m-1} \left(f_{p\theta_j(i)-\theta_{j+1}(i)}^{(1)}\right)^{\tau^{m-1-j}} \equiv 0 \mod \pi^{ml(p-1)\delta+1}.$$
\end{proof}

In order to give the congruence, we introduce a matrix from the base $p$ digits of the minimal solutions associated to $D$ and $p$ (see Definition \ref{basepdigit}).

\begin{definition}
For any elements $\ee,\ee'\in \Sigma$, set 
$$m_{\ee,\ee'}(\Gamma):=\sum_{V=(v_\dd)_D\in V(\ee,\ee')} \prod_D \frac{\gamma_{\dd}^{v_{\dd}}}{v_{\dd}!}$$
Denote by $M(\Gamma)$ the $N\times N$ matrix whose coefficients are the $m_{\ee_i\ee_j}(\Gamma)$.
\end{definition}

With this at hand, the main result of this section is

\begin{proposition}
\label{congmanin}
In the ring $M_\delta$, we have the congruence

$$  \det(\I-TA)\equiv \det\left(\I_N-M(\Gamma)^{\tau^{m-1}}\cdots M(\Gamma)(\pi^{m(p-1)\delta}T)\right)  \mod I_{\delta}    $$

\end{proposition}

\begin{proof}
We first use the decomposition from Lemma \ref{decdetcyc}; it expresses the coefficients of the Fredholm determinants in terms of cyclic minors. Moreover from Lemma \ref{supp1}, when we consider the Fredholm determinant $\det(\I-TA)$ modulo $I_{m\delta}$, we only have to look at the cyclic minors of $A$ having support in $\Sigma$. Fix some injection $\theta:\Z/l\Z \rightarrow \Sigma$, and let $M_\theta$ and $N_\theta$ denote the cyclic minors respectively asociated to the matrices $A[\Sigma]$ and $M(\Gamma)^{\tau^{m-1}}\cdots M(\Gamma)$; we are reduced to show the congruence
$$M_\theta \equiv N_\theta \pi^{lm(p-1)\delta} \mod \pi^{lm(p-1)\delta+1}.$$
We now refine the decomposition of the minor $M_\theta$, using the factorization of $A$. From \cite[Lemma 3.2]{bf} and Lemma \ref{supp2bis}, we have to introduce the set $\II_l(\Sigma)$ of injections from $\Z/l\Z$ to $\Sigma$; reasoning as in the proof of this last Lemma, we get the congruence

$$M_\theta \equiv (-1)^{l-1} \sum\sum \prod_{i=0}^{l-1}\prod_{j=0}^{m-1} \left( \prod_D \lambda^{(1)}_{u_{\dd}^{ij}} \gamma_{\dd}^{u_{\dd}^{ij}}\right)^{\tau^{m-1-j}} \mod \pi^{lm(p-1)\delta+1},$$ 
where the first sum is over $(\theta_1,\cdots,\theta_{m-1})$ in $\II_l(\Sigma)^{m-1}$, and the second over the $(u_{\dd}^{ij})$ indexed by $D\times \Z/l\Z\times \{0,\ldots,m-1\}$, such that 

\begin{itemize}
	\item[(i)] $0\leq u_{\dd}^{ij} \leq p-1$,
	\item[(ii)]  $\sum_D \dd u_{\dd}^{ij}=p\theta_j(i)-\theta_{j+1}(i)$ (with $\theta_0(i)=\theta(i)$, and $\theta_m(i)=\theta(i+1)$);
	\item[(iii)] if we set $u_\dd:=\sum_{i=0}^{l-1} q^{l-1-i}\sum_{j=0}^{m-1} p^{m-1-j} u_{\dd}^{ij}$, then $(u_{\dd})_D\in E_{D,p}(ml)$ is minimal.
\end{itemize}

From (iii), we have $\sum_{i=0}^{l-1}\sum_{j=0}^{m-1} u_{\dd}^{ij}=ml(p-1)\delta$; from (i), and the description of the coefficients $\lambda^{(1)}_i$ for $0\leq i\leq p-1$, we get
$$M_\theta \equiv (-1)^{l-1} \sum\sum \prod_{i=0}^{l-1}\prod_{j=0}^{m-1} \left( \prod_D  \frac{\gamma_{\dd}^{u_{\dd}^{ij}}}{u_{\dd}^{ij}!}\right)^{\tau^{m-1-j}} \pi^{lm(p-1)\delta} \mod \pi^{lm(p-1)\delta+1}.$$ 
From any element $(u_{\dd}^{ij})$ in the second sum, we have constructed a minimal $(u_{\dd})$, whose support is $\varphi$ defined by $\varphi(mi+j)=\theta_j(i)$ from (ii). This is an element of $M(\varphi)$. Conversely, for any $U\in M(\varphi)$, its base $p$ digits satisfy (i), (ii) and (iii) above. 

Thus the second sum is exactly over the base $p$ digits of elements of $M(\varphi)$, and using the bijection $B_\varphi$ in Proposition \ref{decbasep}, we conclude that the second sum is over the set $\prod_{i=0}^{l-1}\prod_{j=0}^{m-1} V(\theta_j(i),\theta_{j+1}(i))$.

It remains to write the cyclic minor $N_\theta$; we use once again \cite[Lemma 3.2]{bf}, and the definition of the matrix as the product $M(\Gamma)^{\tau^{m-1}}\cdots M(\Gamma)$; we get the expression

$$N_\theta=(-1)^{l-1} \sum \prod_{i=0}^{l-1}\prod_{j=0}^{m-1} \left( m_{\theta_j(i),\theta_{j+1}(i)}\right)^{\tau^{m-1-j}},$$

with the sum over $(\theta_1,\cdots,\theta_{m-1})$ in $\II_l(\Sigma)^{m-1}$. Finally we use the description of the coefficients of the matrix $M(\Gamma)$, and we get the expression 
$$N_\theta = (-1)^{l-1} \sum\sum \prod_{i=0}^{l-1}\prod_{j=0}^{m-1} \left( \prod_D  \frac{\gamma_{\dd}^{u_{\dd}^{ij}}}{u_{\dd}^{ij}!}\right)^{\tau^{m-1-j}},$$
where the second sum is over the product $\prod_{i=0}^{l-1}\prod_{j=0}^{m-1} V(\theta_j(i),\theta_{j+1}(i))$. This is the desired result.

\end{proof}

\subsection{The congruence}
\label{zeroes}

We conclude this section by proving Theorem \ref{Lcong}, and giving one of its consequences.

\begin{proof}[Proof of Theorem \ref{Lcong}]
From Lemma \ref{exp} (iii), the function $L(\A^n,f;T)$ is the (alternating) product of the Fredholm determinants $\det(\I-q^{n-k}TA(\Gamma)_I)^{(-1)^{k+1}}$, where $I$ is any subset of $\{1,\cdots,n\}$, $k\leq \# I$, and the above factor appears $\binom{\# I}{k}$ times. From Lemma \ref{supp2}, the series $\det(\I-q^{n-k}TA(\Gamma)_I)$ lies in $M_{\delta_I+n-k}$, which is contained in $M_{\delta}$ from Lemma \ref{DetDI}. Moreover it is invertible in any of these rings, and conguent to $1$ modulo $I_{\delta}$ as long as $\delta_J+n-k>\delta$. Again from Lemma \ref{DetDI}, this happens exactly when $k<\# I$, or $\delta_p(D_I)+n-\# I>\delta$. Thus we get the congruence

$$L(\A^n,f;T)\equiv \prod \det(\I-q^{n-\# I}TA(\Gamma)_I)^{(-1)^{\# I +1}} \mod I_{\delta}$$
where the product is over those $I$ such that $\delta_p(D_I)+n-\#I=\delta_p(D)$. The result is now a direct consequence of Proposition \ref{congmanin}.
\end{proof}

We emphasize a particular case, when we have $\delta_p(D_I)+n-\#I>\delta_p(D)$ for any subset $I\subset \{1,\ldots,n\}$. This is often verified, for instance when $n=1$, or when the set $D$ contains an element with all coordinates positive.

\begin{corollary}
Assume that we have $\delta_p(D_I)+n-\#I>\delta_p(D)$ for any subset $I\subset \{1,\ldots,n\}$; then we have the following congruence in the ring $M_{\delta}$
$$L(\A^n,f;T)^{(-1)^{n+1}}\equiv  \det\left(\I_N-M(\Gamma)^{\tau^{m-1}}\cdots M(\Gamma)(\pi^{m(p-1)\delta}T)\right)  \mod I_{\delta}$$
\end{corollary}

\end{document}